\documentclass[12pt]{amsart}
\usepackage{amssymb, amsmath, amsfonts}

\def\phi{\varphi}

\def\a{\alpha}
\def\b{\beta}
\def\g{\gamma}
\def\G{\Gamma}

\def\l{\lambda}
\def\L{\Lambda}

\def\sp{{\rm span\,}}
\def\Aut{{\sf Aut\,}}

\def\1#1{\overline{#1}}
\def\2#1{\widetilde{#1}}
\def\3#1{\widehat{#1}}
\def\4#1{\mathbb{#1}}
\def\5#1{\mathfrak{#1}}
\def\6#1{{\mathcal{#1}}}

\def\C{{\4C}}
\def\R{{\4R}}

\def\T{{\Theta}}
\newtheorem{theorem}{Theorem}[section]

\newtheorem{definition}[theorem]{Definition}
\newtheorem{lemma}[theorem]{Lemma}

\newtheorem{proposition}[theorem]{Proposition}

\theoremstyle{remark}
\newtheorem{example}[theorem]{Example}

\newtheorem{remark}[theorem]{Remark}

\begin{document}
\numberwithin{equation}{section}

\def\bl{\begin{Lem}}
\def\el{\end{Lem}}
\def\bp{\begin{Pro}}
\def\ep{\end{Pro}}
\def\bt{\begin{Thm}}
\def\et{\end{Thm}}
\def\bc{\begin{Cor}}
\def\ec{\end{Cor}}
\def\bd{\begin{Def}}
\def\ed{\end{Def}}
\def\br{\begin{Rem}}
\def\er{\end{Rem}}
\def\be{\begin{example}}
\def\ee{\end{example}}
\def\bpf{\begin{proof}}
\def\epf{\end{proof}}
\def\ben{\begin{enumerate}}
\def\een{\end{enumerate}}
\def\beq{\begin{equation}}
\def\eeq{\end{equation}}

\title[Rigidity of CR maps ]{Rigidity of CR maps between Shilov boundaries of bounded symmetric
domains}

\author[S.-Y. Kim \& D. Zaitsev]{Sung-Yeon Kim* and Dmitri Zaitsev**}
\address{S.-Y. Kim: Department of Mathematics Education, Kangwon National University, 123 Hyoja-dong, Chuncheon, Kangwon-do, 200-701, Korea }
\email{sykim87@kangwon.ac.kr}
\address{D. Zaitsev: School of Mathematics, Trinity College Dublin, Dublin 2, Ireland}
\email{zaitsev@maths.tcd.ie}

\thanks{*This research was supported by Basic Science Research Program through the National Research Foundation of Korea(NRF) funded by the Ministry of Education, Science and Technology (grant number 2009-0067947)}
\thanks{**Supported in part by the Science Foundation Ireland grant 10/RFP/MTH2878.}

\keywords{bonded symmetric domains, symmetric CR manifold, CR embedding, complete system, totally geodesic embedding}
\subjclass[2000]{
32V40, 32V30, 32V20,
32M05, 53B25, 35N10}

\def\Label#1{\label{#1}}

\begin{abstract}
Our goal is to establish what seems to be the first rigidity result
for CR embeddings between Shilov boundaries of bounded symmetric domains of higher rank.
The result states that any such CR embedding is the standard linear embedding up to CR automorphisms.
Our basic assumption extends precisely the well-known optimal bound for the rank one case.
There are no other restrictions on the ranks, in particular,
the difficult case when the target rank is larger than the source rank is also allowed.
\end{abstract}

\maketitle

\section{Introduction}

Rigidity phenomena for {\em holomorphic isometries}
into complex space forms
go back to Bochner \cite{Bo47} and Calabi \cite{Ca53}
and lead to far going and deeper understanding of {\em metric rigidity}
between general {\em bounded symmetric domains}
in the work of Mok, Mok-Ng and Ng among others.
The reader is referred to the survey by Mok \cite{M11}
for more details, see also the very recent work by Yuan-Zhang \cite{YZ}.
Other important rigidity phenomena for bounded symmetric domains,
such as the strong rigidity of complex structures of their compact quotients
have been discovered by Siu \cite{S80, S81}.

On the other hand, the study of {\em rigidity of holomorphic maps} originated in the work of Poincar\'e \cite{P07} and later Alexander \cite{A74}
for maps sending one open piece of the sphere into another.
It was Webster \cite{W79} who first obtained rigidity 
for holomorphic maps between pieces of spheres of different dimension,
proving that any such map between spheres in $\C^{n}$ and $\C^{n+1}$
is totally geodesic.
Further results in this direction are due to
Faran \cite{Fa86}, Cima-Suffridge \cite{CS83,CS90}, 
Forstneric \cite{F86,F89}
and Huang \cite{H99} who obtained the best known regularity assumption
independent of the dimension difference $n'-n$,
for CR maps between pieces of spheres in $\C^{n+1}$ and $\C^{n'+1}$
under the assumption $n'<2n$. Beyond this bound, 
the rigidity is known to fail as illustrated by the
so-called Whitney map (see e.g. Example 1.1 in \cite{EHZ04}).
 (CR maps are closely related to holomorphic ones,
see e.g.\ \cite{BER99}).
We mention the work by Huang-Ji \cite{HJ01}, Huang \cite{H03} and Huang-Ji-Xu \cite{HJX06}
dealing with this more difficult case, where rigidity 
has to be replaced by the classification of the maps.
On the other note, further rigidity phonemena for CR
maps between real hypersurfaces and  {\em hyperquadrics} have been discovered by
Ebenfelt-Huang and the second author \cite{EHZ04,EHZ05},
Baouendi-Huang \cite{BH05}, Baouendi-Ebenfelt-Huang \cite{BEH09}
and Ebenfelt-Shroff \cite{ES10}.

However, comparing with metric rigidity mentioned above,
holomorphic rigidity for maps between 
{\em bounded symmetric domains $D$ and $D'$ of higher rank}
remains much less understood.
If the rank $r'$ of $D'$ does not exceed the rank $r$ of $D$
and both ranks $r,r'\ge2$, 
the rigidity of {\em proper holomorphic maps} $f\colon D\to D'$
was conjectured by Mok \cite{M89} and proved by Tsai \cite{T93},
showing that $f$ is necessarily totally geodesic (with respect to the Bergmann metric).

The remaining case $r<r'$ seems to be very hard and only little is known.
Tu \cite{Tu02a,Tu02b} established holomorphic rigidity
respectively in the equidimensional case 
(when he proves that the map is biholomorphic and hence $r=r'$) 
and for maps between Cartan type I bounded symmetric domain $D_{p,p-1}$ and $D_{p,p}$
(see below).
Finally, Mok \cite{M08} proved the nonexistence of proper holomorphic maps
between certain pairs of bounded symmetric domains with arbitrary $r'-r$.

The goal of this paper is to take on the rigidity problem for
{\em locally defined CR embeddings} between Shilov boundaries of
general Cartan type I bounded symmetric domains $D_{p,q}$ of higher rank.
This includes the interesting case $r<r'$.
To the best authors' knowledge all
known results on local CR rigidity deal with maps between 
{\em real hypersurfaces}
and rely heavily on {\em Tanaka-Chern-Moser}  approach \cite{Ta,CM} and many of them also on {\em Tanaka-Webster} connection,
which is unavailable for Shilov boundaries of higher rank.
In this paper we follow a new approach 
going back to the general Cartan's moving frame method.
To compensate for the lack of the power of Tanaka-Chern-Moser normalization,
we introduce a sequence of several subsequent adjustments
of moving frames reaching further and further normalization conditions.
We prove:

\begin{theorem}\label{main}
Let $f$ be a smooth CR embedding between open pieces of Shilov boundaries
of two bounded symmetric domains $D_{p,q}$, $D_{p',q'}$ of Cartan type $\rm I$
with $q<p$, $q'<p'$.
Assume that the rank $q>1$ and 
\begin{equation}\Label{main-ineq}
p'-q'<2(p-q).
\end{equation}
Then 
after composing with suitable automorphisms of $D_{p,q}$ and $D_{p',q'}$,
$f$ is given by the block matrix
$$z\mapsto
\begin{pmatrix}
z & 0 \\
0 & I\\
0 & 0
\end{pmatrix}.
$$
\end{theorem}

Note that the assumptions $q<p$ and $q'<p'$ exclude precisely
the cases of square matrices, where one of the Shilov boundaries is totally real
and consequently CR maps are trivial. 
Furthermore, our basic assumption \eqref{main-ineq} corresponds precisely to 
the {\em optimal bound} $n'<2n$  mentioned before in the rank $1$ case ($q=q'=1$) of maps between spheres, where $n=p-1$ and $n'=p'-1$
are the CR dimensions of the spheres. 

\medskip
{\bf Acknowledgement.}
The authors thank the anonymous referee for 
careful reading and helpful remarks.

\section{Preliminaries and adapted frames}\Label{preli}
Throughout this paper we adopt the Einstein summation convention unless mentioned otherwise.
However, if two equal indices appear at the same letter,
e.g.\ $\Phi_{a}^{~a}$, no summation is assumed.
We shall also follow the convention that 
small Greek indices $\a,\b,\g,\delta$ run over $\{1,\ldots,q\}$,
small Latin indices $i, j, k, l$ over $\{1,\ldots,n\}$, 
small Latin indices $a,b,c,d$ over $\{1,\ldots, q'\}$
and large Latin indices $I, J, K, L$ over $\{1,\ldots,n'\}$.
\medskip

Recall that $D_{p,q}$ has the standard realization in the space $\C^{p\times q}$
of $p\times q$ matrices, given by
\begin{equation*}
{D}_{p,q}:=\{z\in \mathbb{C}^{p\times q}:I_q-z^*z \text{ is positive definite} \},
\end{equation*}
where $I_q$ is the identity $q\times q$ matrix and $z^*=\bar z^t$.
The Shilov boundary of $D_{p,q}$ is given by
\begin{equation*}
S_{p,q}=\{z\in \mathbb{C}^{p\times q}:I_q-z^*z=0\}.
\end{equation*}
In particular, $S_{p,q}$ is a symmetric CR manifold of CR dimension $(p-q)\times
q$ in the terminology of \cite{KZ00}. 
For $q=1$, $S_{p,1}$ is the unit sphere in
$\mathbb{C}^p$.
We shall always assume $p>q$ so that $S_{p,q}$
has positive CR dimension, i.e.\ not totally real.

\be\Label{whitney}
The following  generalization of the 
well-known Whitney map
\begin{equation}\Label{whit}
\begin{pmatrix}
z_{11} & \cdots & z_{1q}\cr 
\vdots & \ddots & \vdots\cr
z_{p1} & \cdots & z_{pq}
\end{pmatrix}
\mapsto
\begin{pmatrix}
z_{11} & \cdots & z_{1q'} & 0 & \cdots  &0\cr
\vdots & \ddots & \vdots & \vdots &\ddots & \vdots\cr
z_{p-1,1} & \cdots & z_{p-1,q'} & 0 & \cdots  &0\cr
z_{11}z_{p1} & \cdots & z_{11}z_{pq'} & 0 & \cdots  &0\cr
\vdots & \ddots & \vdots & \vdots &\ddots & \vdots\cr
z_{p1}z_{p1} & \cdots & z_{p1}z_{pq'} & 0 & \cdots  &0\cr
0 & \cdots  &0 &1 & \cdots  &0 \cr
\vdots & \ddots & \vdots & \vdots &\ddots & \vdots\cr
0 & \cdots  &0 &0 & \cdots  &1 \cr
\end{pmatrix}
\in \C^{(p+m)\times (q'+m)}
\end{equation}
restricts to a CR map between the Shilov boundaries,
where $1\le q'\le q$ and $m$ is arbitrary.
This map is not injective in $D_{p,q}$
and hence is not linear after composing
with any automorphisms of $D_{p,q}$ and $D'_{p',q'}$.
For $	q=q'=1$, $m=0$, this is the classical Whitney 
proper map between unit balls in $\C^{p}$ and $\C^{2p-1}$ respectively,
which corresponds to the equality in (\ref{main-ineq})
showing that the latter is an optimal bound.
\ee

\be
The following examples show that there are lots of CR maps between Shilov boundaries
for any choices of ranks $q$ and $q'$.
Fix a collection of proper maps $\phi_{1}, \ldots, \phi_{q'}$
from the unit ball in $\C^{p}$
into unit balls in $\C^{m_{1}},\ldots,\C^{m_{q'}}$ respectively
for any choice of integers $m_{1},\ldots,m_{q'}$.
For any $q$, and any choice of integers $j_{1},\ldots,j_{q'}\in \{1,\ldots,q\}$,
define 
$$\Phi\colon \C^{p\times q} \to \C^{(m_{1}+\ldots+m_{q'})\times q'},$$
such that $\Phi(Z)$ is the block-diagonal matrix
with entries $\phi_{1}(z_{j_{1}}),\ldots,\phi_{n}(z_{j_{q'}})$
on the diagonal.
Then $\Phi$ restricts to a CR map between Shilov boundaries
of the corresponding bounded symmetric domains.
%
\ee


Let $\Aut(S_{p,q})$ be the Lie group of all CR automorphisms of $S_{p,q}$.
By \cite[Theorem~8.5]{KZ00},
every $\phi\in \Aut(S_{p,q})$ extends to a biholomorphic automorphism
of the bounded symmetric domain $D_{p,q}$.
Consider the standard linear inclusion
$$z\mapsto {I_q\choose z}, ~z\in S_{p,q}.$$
Then we may regard
$S_{p,q}$ as a real submanifold in
the Grassmanian $Gr(q,p+q)$
of all $q$-planes in $\C^{p+q}$ and $\Aut(S_{p,q})=\Aut(D_{p,q})$ becomes a
subgroup of the automorphism group of $Gr(q,p+q)$.
In this section we will construct a frame bundle over $S_{p,q}$
associated with the CR structure of $S_{p,q}$ using Grassmannian frames of $Gr(q,p+q)$.

As before, consider the {\em partial CR dimension} $n=p-q$. 
The actual CR dimension of $S_{p,q}$ is $(p-q)q=nq$
and $q=r$ is the rank of the bounded symmetric domain $D_{p,q}$.

For column vectors $u=(u_1,\ldots,u_{p+q})^t$ and
$v=(v_1,\ldots,v_{p+q})^t$ in $\mathbb{C}^{p+q}$,
define the Hermitian inner product by
\begin{equation*}
\langle u,v\rangle:=-(u_1\bar v_1+\cdots+u_q\bar v_q)+(u_{q+1}\bar
v_{q+1}+\cdots+u_{p+q}\bar v_{p+q}).
\end{equation*}
A \emph{Grassmannian frame adapted to} $S_{p,q}$, or simply
$S_{p,q}$-\emph{frame} is a frame $\{ Z_1,\ldots,Z_{p+q}\}$ of
$\mathbb{C}^{p+q}$ with $\det(Z_1,\ldots,Z_{p+q})=1$ such that
\begin{equation}\Label{structure}
\langle Z_\alpha,Z_{q+n+\beta}\rangle=\langle
Z_{q+n+\beta},Z_\alpha\rangle=~\delta_{\alpha\beta},\quad
 \langle Z_{q+j},Z_{q+k}\rangle =\delta_{jk}
\end{equation}
and
\begin{equation}\label{structure2}
\langle Z_\L,Z_\G\rangle=0~\text{   otherwise,   }
\end{equation}
where the capital Greek indices $\L,\G,\Omega$ etc.\ run
from $1$ to $p+q$. We also use the notation
\begin{gather*}
Z:=(Z_1,\ldots,Z_q), \quad
X=(X_1,\ldots,X_n):=(Z_{q+1},\ldots,Z_{q+n}),\quad
Y=(Y_1,\ldots,Y_q):=(Z_{q+n+1}\ldots,Z_{q+p}),
\end{gather*}
so that \eqref{structure} can be rewritten as
\begin{equation}\Label{structure'}
\langle Z_{\a}, Y_{\b}\rangle = \langle Y_{\b}, Z_{\a}\rangle= \delta_{\a\b},
\quad \langle X_{j}, X_{k}\rangle =\delta_{jk},
\end{equation}
i.e.\ the scalar product $\langle\cdot,\cdot\rangle$
in basis $(Z_{\a},X_{j},Y_{\b})$ is given by the matrix
$$
\begin{pmatrix}
0&0& I_{q}\\
0&I_{n}&0\\
I_{q}&0&0\\
\end{pmatrix}.
$$

Let $\mathcal{B}_{p,q}$ be the set of all
$S_{p,q}$-frames. Then $\mathcal{ B}_{p,q}$ can be identified with
$SU(p,q)$ by the left action. The Maurer-Cartan form
$\pi=(\pi_\L^{~\G})$ on $\mathcal{B}_{p,q}$ is given by the equation
\begin{equation}\Label{differential}
dZ_\L=\pi_\L^{~\G}Z_\G,
\end{equation}
where
$\pi$ satisfies the trace-free condition
$$\sum_\L\pi_\L^{~\L}=0$$ and the structure equation
\begin{equation}\Label{struc-eq}
d\pi_\L^{~\G}=\pi_\L^{~\Omega}\wedge\pi_\Omega^{~\G}.
\end{equation}
More explicitly, using the block matrix representation
with respect to the basis $(Z,X,Y)$, we can write
\begin{equation}\Label{pi}
\pi =
\begin{pmatrix}
\pi_{\a}^{~\b} & \pi_{\a}^{~q+j}  & \pi_{\a}^{~q+n+\b}\\
\pi_{q+k}^{~\b} & \pi_{q+k}^{~q+j}  & \pi_{q+k}^{~q+n+\b}\\
\pi_{q+n+\a}^{~\b} & \pi_{q+n+\a}^{~q+j}  & \pi_{q+n+\a}^{~q+n+\b}\\
\end{pmatrix}
= :
\begin{pmatrix}
\psi_{\a}^{~\b} & \theta_{\a}^{~j} & \phi_{\a}^{~\b}\\
\sigma_{k}^{~\b} & \omega_{k}^{~j} & \theta_{k}^{~\b}\\
\xi_{\a}^{~\b} & \sigma_{\a}^{~j} & \3\psi_{\a}^{~\b}\\
\end{pmatrix},
\end{equation}
which satisfies the symmetry relations
\begin{equation}\Label{symmetries}
\begin{pmatrix}
\psi_{\a}^{~\b} & \theta_{\a}^{~j} & \phi_{\a}^{~\b}\\
\sigma_{k}^{~\b} & \omega_{k}^{~j} & \theta_{k}^{~\b}\\
\xi_{\a}^{~\b} & \sigma_{\a}^{~j} & \3\psi_{\a}^{~\b}\\
\end{pmatrix}
=-
\begin{pmatrix}
\3\psi_{\bar\b}^{~\bar\a} & \theta_{\bar j}^{~\bar\a} & \phi_{\bar\b}^{~\bar\a}\\
\sigma_{\bar\b}^{~\bar k} & \omega_{\bar j}^{~\bar k} & \theta_{\bar\b}^{~\bar k}\\
\xi_{\bar\b}^{~\bar\a} & \sigma_{\bar j}^{~\bar\a} & \psi_{\bar\b}^{~\bar\a}\\
\end{pmatrix}
\end{equation}
that follow directly by differentiating \eqref{structure}.

The structure equations \eqref{struc-eq} can be rewritten as
\begin{align}\Label{struct}
d\phi_{\a}^{~\b}&= \psi_{\a}^{~\g}\wedge \phi_{\g}^{~\b}
+\theta_{\a}^{~l}\wedge \theta_{l}^{~\b} + \phi_{\a}^{~\g}\wedge \3\psi_{\g}^{~\b}\\
d\theta_{\a}^{~j}&= \psi_{\a}^{~\g}\wedge \theta_{\g}^{~j}
+\theta_{\a}^{~l}\wedge \omega_{l}^{~j} + \phi_{\a}^{~\g}\wedge \sigma_{\g}^{~j}\\
d\psi_{\a}^{~\b}&= \psi_{\a}^{~\g}\wedge \psi_{\g}^{~\b}
+\theta_{\a}^{~l}\wedge \sigma_{l}^{~\b} + \phi_{\a}^{~\g}\wedge \xi_{\g}^{~\b}\\
d\omega_{k}^{~j}&= \sigma_{k}^{~\g}\wedge \theta_{\g}^{~j}
+\omega_{k}^{~l}\wedge \omega_{l}^{~j} + \theta_{k}^{~\g}\wedge \sigma_{\g}^{~j}\\
d\sigma_{k}^{~\b}&= \sigma_{k}^{~\g}\wedge \psi_{\g}^{~\b}
+\omega_{k}^{~l}\wedge \sigma_{l}^{~\b}
+ \theta_{k}^{~\g}\wedge \xi_{\g}^{~\b}\\
d\xi_{\a}^{~\b}&= \xi_{\a}^{~\g}\wedge \psi_{\g}^{~\b}
+\sigma_{\a}^{~l}\wedge \sigma_{l}^{~\b}
+ \3\psi_{\a}^{~\g}\wedge \xi_{\g}^{~\b},
\end{align}
in particular,
$$d\phi_\alpha^{~\beta}=\theta_\alpha^{~j}\wedge\theta_j^{~\beta}~\text{   mod
}~\phi,$$
where $\phi$ is the span of $\phi_\a^{~\b}$ for all $\a,\b$.

By abuse of notation, we also denote by $Z$
the $q$-dimensional subspace
of $\mathbb{C}^{p+q}$ spanned by $Z_1,\ldots,Z_q$.
Hence $Z$ represents a point in $S_{p,q}$ and vice versa,
any point in $S_{p,q}$ is represented by $Z$
corresponding to an adapted frame $(Z,X,Y)$.
Then $\mathcal{B}_{p,q}$ can be regarded as a
 bundle over ${S}_{p,q}$ via the
projection map $(Z,X,Y)\to Z$.
By another abuse of notation, we shall also use the same letters for the components of $\pi$ and their pullbacks to $S_{p,q}$ via a fixed section.
Note that fixing a section means precisely choosing
an adapted frame $(Z,X,Y)$ at every point $x$ of
(an open subset of) $S_{p,q}$
such that $Z$ represents $x$ as a point in the Grassmanian.

The defining equations
of $S_{p,q}$ can be written as
$$S_{p,q} = \{ [V] \in Gr(q,p+q) : \langle\cdot,\cdot\rangle|_{V}=0 \}$$
and hence their differentiation yields
\begin{equation}\label{diff-eqn}
\langle dZ_\Lambda ,Z_\Gamma\rangle
+\langle
Z_\Lambda,dZ_\Gamma\rangle=0.
\end{equation}
By substituting $dZ_\Lambda=\pi_\Lambda^{~\Gamma}Z_\Gamma$ into $(1,0)$ component of \eqref{diff-eqn} we obtain, in particular,
\begin{equation*}
\phi_\alpha^{~\gamma}\langle Y_\gamma,Z_\beta\rangle~=~
\phi_\alpha^{~\beta}~=~0,
\end{equation*}
when restricted to the $(1,0)$ tangent space.
Comparing the dimensions, we conclude that the kernel of
$\{\phi_\alpha^{~\beta},\alpha,\beta=1,\ldots,q\}$ forms the CR
bundle of $S_{p,q}$, i.e.,
$$\ker(\phi|_Z)=T^{1,0}_Z S_{p,q}\oplus T^{0,1}_Z  S_{p,q}.$$
In other words, $\phi=(\phi_\alpha^{~\beta})$ span the space of contact 
forms on $S_{p,q}$. 
Since $$dZ_\alpha=\psi_\alpha^{~\beta}Z_\beta+\phi_\alpha^{~\beta}Y_\beta+\theta_\alpha^{~j}X_j$$
and $\phi=(\phi_\alpha^{~\beta})$ is a contact form
at $Z=(Z_1,\ldots,Z_q)$, we conlcude that $\phi_{\a}^{~\b}$
and $\theta_{\a}^{~j}$ form together a basis in the space of 
all $(1,0)$ forms.
\medskip

For a change of frame given by
$$
\begin{pmatrix}
\2Z\\
\2X\\
\2Y
\end{pmatrix}
:=U
\begin{pmatrix}
Z\\
X\\
Y
\end{pmatrix}
,$$
$\pi$ changes via
$$\widetilde \pi=dU\cdot U^{-1}+U\cdot\pi\cdot U^{-1}.$$

There are several types of frame changes.

\begin{definition}\Label{changes}
{\rm We call a change of frame}
\begin{enumerate}
\item[i)]change of position {\rm if}
$$
\widetilde Z_\alpha=W_\alpha^{~\beta}Z_\beta,\quad
\widetilde Y_\alpha=V_\alpha^{~\beta}Y_\beta,\quad
\widetilde X_j=X_j,
$$
{\rm where $W=(W_\alpha^{~\beta})$ and $V=(V_\alpha^{~\beta})$ are
$q\times q$ matrices satisfying $V^*W=I_q$};

\item[ii)]change of real vectors {\rm if}
$$
\widetilde Z_\alpha=Z_\alpha,\quad
\widetilde X_j=X_j,\quad
\widetilde Y_\alpha=Y_\alpha+H_\alpha^{~\beta}Z_\beta,
$$
or
\begin{equation}
\begin{pmatrix}
\2Z_{\a}\\
\2X_{j}\\
\2Y_{\a}
\end{pmatrix}
=
\begin{pmatrix}
I_{q} & 0 & 0\\
0 & I_{n} & 0\\
H_{\a}^{~\b}&0& I_{q}
\end{pmatrix}
\begin{pmatrix}
Z_{\b}\\
X_{k}\\
Y_{\b}
\end{pmatrix},
\end{equation}
{\rm where $H=(H_\alpha^{~\beta})$ is a hermitian matrix};

\item[iii)]dilation {\rm if}
$$
\widetilde Z_\alpha=\lambda_{\alpha}^{-1}Z_\alpha,\quad
\widetilde Y_\alpha=\lambda_\alpha Y_\alpha,\quad
\widetilde X_j=X_j,
$$
{\rm where $\lambda_\alpha>0$};

\item[iv)]rotation {\rm if}
$$
\widetilde Z_\alpha=Z_\alpha,\quad
\widetilde Y_\alpha=Y_\alpha,\quad
\widetilde X_j=U_j^{~k}X_k,
$$
{\rm where $(U_j^{~k})$ is a unitary matrix.}
\end{enumerate}
\end{definition}
\medskip

Consider a change of position 
as in Definition~\ref{changes}.
Then
$\phi$ and
$\theta$ change to
$$
\widetilde
\phi_\alpha^{~\beta}=W_\alpha^{~\gamma}\phi_\gamma^{~\delta}W^{*}{}_{\delta}^{~\b},
\quad W^{*}{}_{\delta}^{~\b}=\overline{W_{\beta}^{~\delta}},\quad
\widetilde\theta_\alpha^{~j}=W_\alpha^{~\beta}\theta_\beta^{~j}.
$$

We shall also make use of the change of frame given by
$$
\widetilde Z_\alpha=Z_\alpha,\quad
\widetilde X_j=X_{j} + C_j^{~\beta}Z_\beta,\quad
\widetilde Y_\alpha=Y_\alpha+A_\alpha^{~\beta}Z_\beta+B_\alpha^{~j}X_j,
$$
or
\begin{equation}
\begin{pmatrix}
\2Z_{\a}\\
\2X_{j}\\
\2Y_{\a}
\end{pmatrix}
=
\begin{pmatrix}
I_{q} & 0 & 0\\
C_{j}^{~\b} & I_{n} & 0\\
A_{\a}^{~\b}& B_{\a}^{~j}& I_{q}
\end{pmatrix}
\begin{pmatrix}
Z_{\b}\\
X_{k}\\
Y_{\b}
\end{pmatrix},
\end{equation}
such that
$$C_j^{~\alpha}+B_j^{~\alpha}=0$$
and
$$(A_\alpha^{~\beta} + \overline{A_\beta^{~\alpha}})
+B_\alpha^{~j}B_j^{~\beta}=0,$$
where
$$B_j^{~\alpha}:=\overline{B_\alpha^{~j}}.$$
Then the new frame $(\widetilde Z,\widetilde Y,\widetilde X)$ is an $S_{p,q}$-frame. In fact,
\begin{multline}
0=\langle \2Y_{\a},\2Y_{\b}\rangle=
\langle Y_\alpha+A_\alpha^{~\delta}Z_\delta+B_\alpha^{~j}X_j,
Y_\b+A_\b^{~\g}Z_\g+B_\b^{~k}X_k \rangle \\
=
 A_{\a}^{~\b} \langle Z_{\b}, Y_{\b}\rangle
 + \1{A_{\b}^{~\a}} \langle Y_{\a}, Z_{\a}\rangle
 +\sum_{j}B_{\a}^{~j}\1{B_{\b}^{j}}\langle X_{j}, X_{j}\rangle
 = (A_{\a}^{~\b} + \1{A_{\b}^{~\a}})+ \sum_{j }B_{\a}^{~j}\1{B_{\b}^{j}},
\end{multline}
and
\begin{multline}
0=\langle \2X_{j}, \2Y_{\a}\rangle=
\langle X_{j} + C_j^{~\beta}Z_\beta ,
Y_\alpha+A_\alpha^{~\delta}Z_\delta+B_\alpha^{~k}X_k
\rangle 
=
 C_{j}^{~\a} \langle Z_{\a}, Y_{\a}\rangle
 + \1{B_{\a}^{~j}} \langle X_{j}, X_{j}\rangle
  = C_{j}^{~\a}
 + \1{B_{\a}^{~j}},
\end{multline}
whereas the other scalar products are obviously zero.
Furthermore, we claim that the related $1$-forms $\widetilde\phi_\alpha^{~\beta}$ remain the same, while $\widetilde\theta_\alpha^{~j}$ change to
$$\widetilde\theta_\alpha^{~j}=\theta_\alpha^{~j}-\phi_\alpha^{~\beta}B_\beta^{~j}.$$
Indeed, differentiation yields
$$d\2Z_{\a} =
\2\psi_{\a}^{~\b}\2Z_{\b} + \2\theta_{\a}^{j}\2X_{j} +  \2\phi_{\a}^{~\b} \2Y_{\b}  =
\2\psi_{\a}^{~\b}Z_{\b} + \2\theta_{\a}^{~j}(X_{j} + C_j^{~\beta}Z_\beta) +  \2\phi_{\a}^{~\b} (Y_\b+A_\b^{~\g}Z_\g+B_\b^{~j}X_j)
 $$
$$
=dZ_{\a}=\psi_{\a}^{~\b}Z_{\b} + \theta_{\a}^{j}X_{j} +  \phi_{\a}^{~\b} Y_{\b}
$$
and the claim follows from identifying the coefficients.
\medskip

\section{Cartan's Lemma}
We shall routinely use the Cartan's Lemma
for complex-valued forms:

\begin{lemma}[Cartan's Lemma]
Let $\theta_{1},\ldots,\theta_{r}$
be complex-linearly independent
 complex-valued $1$-forms on a real manifold $M$ 
 and $\phi_{1},\ldots,\phi_{r}$ be further
 complex-valued $1$-forms on $M$ satisfying
\begin{equation}\Label{cartan-eq}
\theta_{1}\wedge\phi_{1}+\ldots+\theta_{r}\wedge\phi_{r}=0.
\end{equation}
Then 
$$\phi_{j}=0 \mod \{ \theta_{1},\ldots, \theta_{r}\}$$
for each $j=1,\ldots,r$.
\end{lemma}

\begin{proof}
Complete $\theta_{1},\ldots,\theta_{r}$
to a basis $\theta_{1},\ldots, \theta_{s}$
in the space of all complex-valued $1$-forms on $M$.
Then we can write
$$\phi_{j}=c_{j}^{~k}\theta_{k}$$
for suitable coefficients $c_{j}^{~k}$.
Then substituting into \eqref{cartan-eq}, using the fact
that the set of $\theta_{i}\wedge\theta_{j}$ with $i<j$
is a basis in the space of all $2$-forms,
and identifying coefficients of $\theta_{j}\wedge\theta_{k}$
for $j\le r$, $k>r$, we conlcude 
$$c_{j}^{~k}=0, \quad j\le r<k,$$
and the claim follows.
\end{proof}

\section{Determination of $\Phi_{a}^{~b}$ and $\Theta_{a}^{~J}$
modulo $\phi$ using the Levi form identities.}

Let $p>q$, $p'>q'$ be positive integers and let $f$ be a local CR embedding 
from $S_{p,q}$ into  $S_{p',q'}$ 
Denote by
$\mathcal{B}_{p,q}$ and $\mathcal{B}_{p',q'}$ the Grassmannian frame
bundles adapted to $S_{p,q}$ and $S_{p',q'}$ respectively.
We set $n:=p-q$, $n':=p'-q'$ and follow the index convention 
at the beginning of \S\ref{preli}.

We shall consider the connection forms
$\phi_\alpha^{~\beta}$, $\theta_\alpha^{~j}$, $\psi_\alpha^{~\beta}$, $\omega_j^{~k}$, $\sigma_{j}^{~\b}$, $\xi_{\a}^{~\b}$
on $\mathcal{B}_{p,q}$ pulled back to $S_{p,q}$ 
and  denote by capital letters
$\Phi_a^{~b}$, $\Theta_a^{~J}$, $\Psi_a^{~b}$, $\Omega_{J}^{~K}$, $\Sigma_{K}^{~b}$, $\Xi_{a}^{~b}$ their corresponding counterparts on $\mathcal{B}_{p',q'}$ pulled back to $S_{p',q'}$.
Furthermore, we shall adopt the convention that
any form is assumed to be zero whenever its indices
are out of the range where the form is defined,
e.g. $\theta_{a}^{~J}=0$ if either $a>q$ or $J>n$,
or $\phi_{a}^{~b}=0$ if either $a>q$ or $b>q$.

Since $\phi=(\phi_\alpha^{~\beta})$ and $\Phi=(\Phi_a^{~b})$ are
contact forms on $S_{p,q}$ and $S_{p',q'}$, respectively, the
pull back of $\Phi$ via $f$ is a linear
combination of $\phi=(\phi_{\a}^{~\b})$.

We shall abuse the notation by writing
$\Sigma$ instead of $f^{*}\Sigma$ for any form $\Sigma$ on $S_{p',q'}$.
Thus all our forms will be understood on $S_{p,q}$
and any form on $S_{p',q'}$ will be assumed pulled back to $S_{p,q}$ via the given CR map $f$ without explicit mentioning.

In this section our analysis will be based on using the structure equation
for $\phi$ modulo the ideal generated by the contact forms $\phi_{\a}^{~\b}$, i.e.\ on the equations
\begin{equation}
d\phi_{a}^{~b} = \theta_{a}^{~j}\wedge \theta_{j}^{~b} \mod \phi, \quad
d\Phi_{a}^{~b} = \Theta_{\a}^{~J}\wedge \Theta_{J}^{~b} \mod \phi. \Label{seq}
\end{equation}
By writing identities modulo $\phi$ we shall always mean that the difference between the left- and right-hand sides is contained in the ideal generated by the components $\phi_{\a}^{~\b}$ in the exterior algebra.
In the second identity we have also used the fact mentioned above that
(the pullback of) any $\Phi_{a}^{~b}$ is a linear combination of $\phi_{\a}^{~\b}$.
Note that due to our convention, both sides of the first equation are zero
 if either $a>q$ or $b>q$ and for the same reason the summation is only performed over $j\in\{1,\ldots,n\}$.

\subsection{Determination of $\Phi_{1}^{~1}$}
Consider the diagonal terms $\Phi_a^{~a}$, $a=1,\ldots,q'$.
Suppose that (the pullbacks of) $\Phi_a^{~a}$
vanish identically for all $a$. Then \eqref{seq} yields
$$ 0=d\Phi_a^{~a}= -\sum_{J} \Theta_a^{~J}\wedge  \1{\Theta_a^{~J}}
~\text{    mod    }~\phi.$$
Since each $\Theta_a^{~J}$ is a $(1,0)$ form and each wedge product is non-negative on $(T,\bar T)$ where $T$ is any $(1,0)$ vector, it follows that
$$\Theta_a^{~J}=0~\text{   mod   }~\phi,$$
which contradicts the assumption that $f$ is an embedding.

Hence there exists at least one diagonal
term of $\Phi$ whose pullback does not vanish identically.
Choose such a diagonal term of $\Phi$, say $\Phi_1^{~1}$.
Then on an open set, $\Phi_1^{~1}\neq 0.$
Since the pullback of $\Phi_{1}^{1}$ to $S_{p,q}$ is a contact form, we can write
$$\Phi_1^{~1}= c_\alpha^{~\beta}\phi^{~\alpha}_{\beta}$$
for some smooth functions $c_\alpha^{~\beta}$. Since
$(\phi_\alpha^{~\beta})$ and $(\Phi_a^{~b})$ are antihermitian, the
matrix $(c_\alpha^{~\beta})$ is hermitian. Then there exists
a change of frame on $S_{p,q}$ (change of position in Definition~\ref{changes}) given by
$$\widetilde Z_\alpha=U_\alpha^{~\beta}Z_\beta,
~\widetilde Y_\alpha=U_\alpha^{~\beta}Y_\beta,~\widetilde X_j=X_j,$$ for some unitary matrix $U$
such that $c_{\a}^{\b}$ is diagonalized and hence
the new contact forms $\phi_\alpha^{~\beta}$, $\alpha,\beta=1,\ldots,q$, satisfy
\begin{equation*}
\Phi_1^{~1}=\sum_{\alpha=1}^rc_\alpha \phi_\alpha^{~\alpha},
\quad 1\le r\le q,
\end{equation*}
where $c_\alpha$, $\alpha=1,\ldots,r$, are nonzero real valued smooth functions.
Then \eqref{seq} yields
\begin{equation}\Label{phi11}
\sum_{J} \Theta_{1}^{~J}\wedge \1{\Theta_{1}^{~J}}
= \sum_{\a,j} c_{\a} \theta_{\a}^{~j}\wedge \1{\theta_{\a}^{~j}} \mod \phi,
\end{equation}
which implies $c_{\a}> 0$
in view of the non-negativity mentioned above since the forms $\theta_{\a}^{~\g}\wedge \1{\theta_{\a}^{~\g}}$ are linearly independent.
 Hence after dilation of $\Phi_1^{~1}$, we may assume that $$c_1=1.$$

\begin{lemma}\Label{r}
Assuming $n'<2n$, we have  $r=1$ and
\begin{align}
\Phi_{1}^{~1}&=\phi_{1}^{~1},\Label{phi11-norm} \\
\Theta_1^{~J}&=\theta_{1}^{~J} \mod \phi.\Label{theta-norm}
\end{align}
\end{lemma}

\bpf
Let
\begin{equation}\Label{h-eq}
\Theta_1^{~J}=h^{J,\alpha}_{~j}\theta_\alpha^{~j}~\text{   mod   }~\phi.
\end{equation}
Then \eqref{phi11} implies
$$\sum_J h_{~j}^{J,\alpha}\1{h_{~k}^{J,\beta}}
=c_\alpha\delta_{\alpha\beta}\cdot\delta_{jk},$$
where $c_{\a}:=0$ for $\a>r$.
Thus the vectors $h^{\a}_{~j}:=(h_{~j}^{1,\a},\ldots,h_{~j}^{n',\a})$ are pairwise orthogonal and have length $c_{\a}$ independent of $j$.
Therefore after a suitable
rotation (see Definition~\ref{changes})
$$\widetilde \Theta_{a}^{~J}=\Theta_{a}^{~K}U_K^{~J},$$
where $(U_K^{~J})$ is unitary,
we may assume that $h^{1}_{~j}$, whose length is $c_{1}=1$, are precisely the first $n$ standard unit vectors in
$\C^{n'}$, i.e.\
\begin{equation}
h^{J,1}_{~j}=\delta_{Jj}.
\end{equation}
 Then for every fixed $\a\ne1$, we have $n$ orthogonal vectors $h^{\a}_{~j}$ in the span of the last $n'-n$ standard unit vectors.
Since $n'-n<n$ by our assumption, the latter is only possible when
$h^{\a}_{~j}=0$ for all $\a\ne1$. Thus we obtain
\begin{equation}
h^{J,\alpha}_{~j} = \delta_{Jj}\delta_{1\a}.
\end{equation}
Then \eqref{h-eq} implies \eqref{theta-norm}
 and  hence \eqref{phi11} implies $r=1$ and therefore
 \eqref{seq} implies \eqref{phi11-norm}.
\epf

\subsection{Determination of $\Phi_{2}^{~2}$ and $\Phi_{2}^{~1}$}
Consider the ideal $\theta_\alpha$ generated by
$\theta_\alpha^{~j}$ for $1\leq j\leq n$.
Let
\begin{equation}\Label{la}
\Phi_a^{~1}=\lambda_a\phi_1^{~1}~\text{mod}~
\{\phi_\alpha^{~\beta},  \a\ge 2 \text{ or } \b\ge2\}, \quad a\ge2,
\end{equation}
for some smooth functions $\lambda_a$, $a=2,\ldots,q'$.
Then \eqref{seq}  together with Lemma~\ref{r} imply
\begin{equation}\Label{thaj}
\Theta_a^{~j}\wedge \theta_j^{~1}=\lambda_a\theta_1^{~j}\wedge\theta_j^{~1}~
\mod \{\theta_\alpha,~ \1{\theta_{\alpha}},~\alpha\ge2\},~\phi, \quad a\ge2.
\end{equation}
Then there exists a change of position that leaves $\Theta_{1}^{~J}$ invariant and replaces $\Theta_a^{~J}$
with $\Theta_a^{~J}-\lambda_a\Theta_1^{~J}$, $a\ge2$, (see the discussion after Definition~\ref{changes}).
The same change of position leaves $\Phi_{1}^{~1}$ invariant
and transforms $\Phi_{a}^{~1}$ into $\Phi_{a}^{~1}-\l_{a}\Phi_{1}^{~1}$ for $a\ge2$.
After performing such change of position, \eqref{la} becomes
$$\Phi_a^{~1}=0~\text{mod}~
\{\phi_\alpha^{~\beta} : \a\ge 2 \text{ or } \b\ge2\}, \quad a\ge2,$$
and \eqref{thaj} becomes
\begin{equation}
\Theta_a^{~j}\wedge \theta_j^{~1}=0~ \mod \{\theta_\alpha,~ \1{\theta_{\alpha}} :\alpha\ge2\},~\phi, \quad a\ge2.
\end{equation}
Since $\Theta_a^{~j}$ are $(1,0)$ but $\theta_j^{~1}$ are $(0,1)$ and linearly independent,
it follows that
\begin{equation}\Label{theta-aj}
\Theta_a^{~j}=0 \mod  \{\theta_\a : \a\ge2\},\, \phi, \quad a\geq 2.
\end{equation}

Next for each $a\ge2$, let
\begin{equation}\Label{lab}
\Phi_a^{~a}=\lambda_{a,\beta}\phi_1^{~\beta}
\mod\{\phi_\alpha^{~\g} : \alpha\geq 2\}
\end{equation}
for some
functions $\lambda_{a,\beta}$. Suppose there
exists $a$ and $\beta$ such that $\lambda_{a,\beta}\neq 0$. We may assume $a=2$.
Using the identity
$$
d\Phi_2^{~2}=~\Theta_2^{~J}\wedge\Theta_J^{~2} \mod \Phi$$
together with \eqref{theta-aj} we obtain
\begin{equation}\Label{2a}
\sum_{J=n+1}^{n'} \Theta_2^{~J}\wedge \Theta^{~2}_J=\lambda_{2,\beta}~ \theta_1^{~j}\wedge\theta_j^{~\beta} \mod \{\theta_\a: \a\ge2\},\phi,
\end{equation}
where $\l_{2,\b}\ne0$.
On the left-hand side we have a linear combination of $n'-n$ $(1,0)$ forms,
whereas on the right-hand side we have a linear combination of at least $n$ linear independent $(1,0)$ forms with nonzero coefficients.
Since $n'-n<n$, this is impossible.
Hence we have $\l_{a,\b}=0$ for all $a,\b$ and therefore \eqref{lab} implies
$$\Phi_a^{~a}=0~\text{   mod
}\{\phi_\alpha^{~\beta} : \alpha\ge 2\}, \quad a\ge2.$$
Since $\Phi_{a}^{~b}$ and $\phi_{\a}^{~\b}$ are antihermitian, we also have
$$\Phi_a^{~a}=0~\text{   mod
}\{\phi_\alpha^{~\beta} : \b\ge 2\}, \quad a\ge2,$$
and hence
\begin{equation}\Label{phi-aa}
\Phi_a^{~a}=0~\text{   mod
}\{\phi_\alpha^{~\beta}: \a,\b\ge2 \}, \quad a\ge 2.
\end{equation}
Now \eqref{seq} implies
\begin{equation}\Label{aa}
\sum_{J=n+1}^{n'} \Theta_a^{~J}\wedge \1{\Theta_{a}^{~J}}=0~\text{   mod   }\{\theta_\a : \a\ge2\},\phi, \quad a\ge2,
\end{equation}
which implies
\begin{equation}\Label{aj}
\Theta^{~J}_a=0 ~\text{   mod   }\{\theta_\a : \a\ge2\},\phi,
\quad a\ge2,\, J>n.
\end{equation}
Together with \eqref{theta-aj} this yields
\begin{equation}\Label{theta-aJ}
\Theta_a^{~J}=0 \mod  \{\theta_\a : \a\ge2\},\, \phi, \quad a\geq 2.
\end{equation}

Now we redo our procedure for $\Phi_{a}^{~b}$. We can write
\begin{equation}\Label{phi-ab'}
\Phi_{a}^{~b}= \l_{a}^{~b}{}_{\b}^{~\a} \phi_{\a}^{~\b}
\end{equation}
for which \eqref{seq} yields
$$\Theta_a^{~J}\wedge \Theta_J^{~b}=
\l_{a}^{~b}{}_{\b}^{~\a} \theta_\a^{~j}\wedge\theta_j^{~\b}~
\text{mod}~\phi.$$
Then substituting \eqref{theta-aJ} we obtain
\begin{equation}
\l_{a}^{~b}{}_{\b}^{~\a} \theta_\a^{~j}\wedge\theta_j^{~\b} =0
\mod \{\theta_\g^{~k}\wedge\theta_{l}^{~\delta}  : \g,\delta\ge 2\},\phi, \quad a,b\ge2,
\end{equation}
which implies
\begin{equation}
\l_{a}^{~b}{}_{1}^{~\a} = \l_{a}^{~b}{}_{\b}^{~1} = 0, \quad a,b\ge 2.
\end{equation}
Hence \eqref{phi-ab'} yields
\begin{equation}
\Phi_{a}^{~b}= 0 \mod \{\phi_\alpha^{~\beta}: \a,\b\ge2 \}, \quad a,b\ge 2.
\end{equation}
Summarizing we obtain the following:
\begin{align}
\Phi_a^{~1}&=0 \mod
\{\phi_\alpha^{~\beta} : \a\ge 2 \text{ or } \b\ge2\}, \quad a\ge2, \\
\Phi_{a}^{~b}&= 0 \mod \{\phi_\alpha^{~\beta}: \a,\b\ge2 \}, \quad a,b\ge 2,\\
\Theta^{~J}_a &=0 \mod \{\theta_\a: \a\ge2\},\phi, \quad a\ge2.\Label{before-repeat}
\end{align}

Now repeat the argument from the beginning of this section and assume first that
$\Phi_{a}^{~a}=0$ for all $a\ge 2$.
We obtain
$$\Theta_{a}^{~J}=0 \mod \phi, \quad a\ge2,$$
and hence $df$ vanishes on the kernel of all $\theta_{1}^{~j}$ and $\phi_{\a}^{\b}$.
Since $f$ is an embedding, it follows that the latter kernel
equals the full complex tangent space, i.e.\ $q=1$.
In this case \eqref{aj} implies
$$\Theta^{~J}_a=0 \mod \phi, \quad a>1=q.$$
Then writing \eqref{phi-ab'}  and proceeding as before
we obtain $\l_{a}^{~b}{}_{\b}^{~\a}=0$ and hence
\begin{equation}
\Phi_{a}^{~b}=0, \quad q=1.
\end{equation}
(Note that we have assumed $q\ge 2$ excluding this case.
However, we shall repeat this procedure when
a similar case will occur.)

In the remaining case $q>1$, our assumption above cannot hold,
i.e.\ $\Phi_{a}^{~a}\ne 0$ for some $a$, say $a=2$.
Then \eqref{phi-aa} implies that, after a change of position as before,
we may assume that
$$\Phi_{2}^{~2}=\sum_{\a\ge 2} c_{\a} \phi_{\a}^{~\a}$$
for some $c_{\a}\ge0$ not all zero.
Then \eqref{seq} yields
\begin{equation}\label{two theta1}
\Theta_2^{~J}\wedge \Theta_J^{~2}=\sum_{\alpha\ge2} c_\alpha\theta_\alpha^{~j}\wedge\theta_j^{~\alpha}
~\text{   mod   }~\phi.
\end{equation}
Since the proof of Lemma~\ref{r} can be repeated for $\Phi_{2}^{~2}$
instead of $\Phi_{1}^{~1}$, we conclude that
the rank of the left-hand side of \eqref{two theta1} restricted
to the complex tangent space is $n$.
Therefore, in the right-hand side, only one $c_{\a}$, say $c_{2}$ can be different from zero. After a dilation (see Definition~\ref{changes}), we may assume
$$\Phi_{2}^{~2}= \phi_{2}^{~2}$$
and hence
\begin{equation}\Label{theta2}
\sum_{J}\T_{2}^{~J}\wedge \1{\T_{2}^{~J}} = \sum_{j}\theta_{2}^{~j}\wedge \1{\theta_{2}^{~j}}
\mod\phi.
\end{equation}

We claim that each $\T_{2}^{~J}$ is a linear combination of only $\theta_{2}^{~j}$ modulo $\phi$. Indeed, if $\T_{2}^{~J}$ were a combination of $\theta_{\a}^{~j}$ modulo $\phi$, where some of them
enters with a nonzero coefficient $\l_{\a}$ with $\a\ne 2$,
we would have $\theta_{\a}^{~j}\wedge \1{\theta_{\a}^{~j}}$
entering with positive coefficient
$\ge\l_{\a}\1{\l_{\a}}$ in the right-hand side of \eqref{theta2}, which is impossible, proving our claim.
As in the proof of Lemma~\ref{r} we now write
\begin{equation}\Label{teta2}
\Theta_2^{~J}=h^{J}_{~j}\theta_2^{~j}~\text{   mod   }~\phi.
\end{equation}
Since
\begin{equation}\Label{tet2}
\Phi_{2}^{~1}= \l_{\b}^{~\a} \phi_{\a}^{~\b}
\end{equation}
for suitable $\l_{\a}^{~\b}$, we obtain
\begin{equation}
\T_{2}^{~J}\wedge \T_{J}^{~1} = \l_{\b}^{~\a} \theta_{\a}^{~j} \wedge
\theta_{j}^{~\b} \mod \phi,
\end{equation}
which in view of \eqref{teta2} and Lemma~\ref{r}, yields
\begin{equation}\Label{h-id}
h^{k}_{~j}\theta_2^{~j} \wedge \theta_{k}^{~1} = \l_{\b}^{~\a} \theta_{\a}^{~j} \wedge
\theta_{j}^{~\b} \mod \phi.
\end{equation}
Since the right-hand side contains no terms
$\theta_{\a}^{~j} \wedge
\theta_{k}^{~\b}$ with $j\ne k$, it follows that
$h^{k}_{~j}=0$ for $j\ne k$ and hence $h^{j}_{~j}=\l_{2}^{~1}=:\l$ for all $j$
and $\l_{\b}^{~\a}=0$ for $(\a,\b)\ne(1,2)$.
Then \eqref{teta2} implies
\begin{equation}\Label{teta3}
\Theta_2^{~j}=\l \theta_2^{~j}~\text{   mod   }~\phi.
\end{equation}
Finally, substituting \eqref{teta3} into \eqref{theta2} and identifying coefficients we obtain
$$ \l\bar\l \delta_{ij} + \sum_{J>n} h_{~i}^{J}\1{h_{~j}^{J}}
=\delta_{ij}.$$
In particular, it follows that the vectors $h_{i}:=(h^{n+1}_{~i},\ldots,h^{n'}_{~i})$ are orthogonal and of the same length.
But since we have assumed $n'-n<n$, we must have $h_{i}=0$
and therefore $|\l|=1$. Now we perform a change of position
as in Definition~\ref{changes} with $W_{\a}^{~\b}:=c_{\a}\delta_{\a\b}$ with
$c_{\a}=1$ for $\a\ne2$ and $c_{2}=\l$.
Then we arrive at the following relations:
\begin{equation}
\Phi_{a}^{~a}=\phi_{a}^{~a}, \quad a=1,2,
\end{equation}
\begin{equation}
 \T_{a}^{~J}=\theta_{a}^{~J} \mod \phi, \quad a=1,2.
\end{equation}

Since after the last change of position, we have $\l=1$ in \eqref{teta3},
we obtain from \eqref{h-id} that $\l_{\b}^{~\a}=0$ unless $\a=2$ and $\b=1$,
in which case $\l_{2}^{~1}=1$. Then substituting into \eqref{tet2} yields
\begin{equation}
\Phi_{2}^{~1}=\phi_{2}^{~1}.
\end{equation}

\subsection{Determination of $\Phi_{a}^{~b}$}
Now we repeat again the arguments
after the proof of Lemma~\ref{r},
where we replace $1$ by $2$ and $2$ by $3$, to arrive at the identities:
\begin{align}
\Phi_a^{~2}&=0 \mod
\{\phi_\alpha^{~\beta} : \a\ge 3 \text{ or } \b\ge3\}, \quad a\ge3, \\
\Phi_{a}^{~b}&= 0 \mod \{\phi_\alpha^{~\beta}: \a,\b\ge3 \}, \quad a,b\ge 3,\\
\Theta^{~J}_a &=0 ~\text{   mod   } \{\theta_\a : \a\ge 3\},\,\phi, \quad a\ge3.
\end{align}
Then continuing following the arguments after \eqref{before-repeat}
with the same replacements, we obtain
\begin{equation}
\Phi_{3}^{~\a}=\phi_{3}^{~\a},\quad \a=1,2,3,
\end{equation}
\begin{equation}
 \T_{3}^{~J}=\theta_{3}^{~J}\mod \phi.
\end{equation}

Finally, arguing by induction on $b=4,\ldots,q'$, and proceeding
by repeating the same arguments with
$1$ replaced by $b$ and $2$ by $b+1$,
we obtain the following lemma.

\begin{lemma}\Label{eds}
For any local CR embedding $f$ from $S_{p,q}$ into $S_{p',q'}$,
there is a choice of sections of the bundles
$\6B_{p,q}\to S_{p,q}$ and $\6B_{p',q'}\to S_{p',q'}$
such that the pulled back forms satisfy
\begin{align*}
\Phi_a^{~b}-\phi_a^{~b}&=0,\\
\Theta_a^{~J}-\theta_a^{~J}&=0 \mod\phi.
\end{align*}
\end{lemma}

\begin{remark}\Label{back1}
A change of section of $\6B_{p,q}\to S_{p,q}$
(corresponding to a change of frame on $S_{p,q}$)
has been used in course of the proof.
However, once Lemma~\ref{eds} has been established,
one can change the frame on $S_{p,q}$ back to the original one
together with the corresponding change of the frame on $S_{p',q'}$
involving only the subframe $(Z_a, X_J, Y_b)$ with $a,b\le q, J\le n$,
such that the conclusion of the lemma remains valid.
\end{remark}

\section{Determination of $\Theta_{a}^{~J}$.}
In our next analysis we shall use
the full structure equations for $\phi$ and $\Phi$
which in view of Lemma~\ref{eds} take the form
\begin{align}
d\phi_{a}^{~b} &=
\psi_a^{~\g}\wedge\phi_\g^{~b}  +
\theta_a^{~j}\wedge\theta_j^{~b}+
\phi_{a}^{~\g}\wedge \3\psi_{\g}^{~b},\Label{struc20}\\
d\phi_{a}^{~b} &=
\Psi_a^{~\g}\wedge\phi_\g^{~b}  +
\Theta_a^{~J}\wedge\Theta_J^{~b}+
\phi_{a}^{~\g}\wedge \3\Psi_{\g}^{~b},\Label{struc2}
\end{align}
and    
their difference
\begin{equation}\Label{struc-phi}
(\Psi_a^{~\g} - \psi_{a}^{~\g})\wedge\phi_\g^{~b}  +
\Theta_a^{~J}\wedge\Theta_J^{~b}
-\theta_a^{~j}\wedge\theta_j^{~b}+
\phi_{a}^{~\g}\wedge(\3\Psi_{\g}^{~b} - \3\psi_{\g}^{~b})=0,
\end{equation}
as well as the structure equations for $\theta$ and $\Theta$:
\begin{align}
d\theta_{a}^{~J} &=
\psi_a^{~\b}\wedge \theta_{\b}^{~J}
+\theta_a^{~k}\wedge \omega_k^{~J}
+\phi_{a}^{~\b} \wedge \sigma_{\b}^{~J},
\Label{dtheta0-1}\\
d\Theta_{a}^{~J} &=
\Psi_a^{~b}\wedge \Theta_b^{~J}
+\Theta_a^{~K}\wedge \Omega_K^{~J}
+\phi_{a}^{~\b} \wedge \Sigma_{\b}^{~J}.
\Label{dtheta0}
\end{align}


Our next goal is to determine $\Theta_{a}^{~J}$.
It will be determined together with
components $\Psi_{a}^{~b}$
and $\Omega_{K}^{~J}$ modulo $\phi$.
In view of Lemma~\ref{eds} we can write
\begin{align}
\Theta_a^{~J}-\theta_a^{~J}=
\eta_{a}^{~J}{}_{\b}^{~\g}   \phi_\g^{~\b},\Label{eq1}
\end{align}
for some $\eta_{a}^{~J}{}_{\b}^{~\g}$, and using the symmetry relations
\eqref{symmetries},
\begin{align}
\Theta_J^{~a}-\theta_J^{~a}=
\eta_{J}^{~a}{}_{\g}^{~\b}   \phi_\b^{~\g},\Label{eq1'}
\end{align}
where
\begin{equation}
\eta_{J}^{~a}{}_{\g}^{~\b}:= \1{\eta_{a}^{~J}{}_{\b}^{~\g}}.
\end{equation}

We will show that after a frame change, we may assume that
$$\eta_{a}^{~J}{}_{\b}^{~\g}=0.$$
Using \eqref{eq1} we compute
\begin{equation}
\Theta_a^{~J}\wedge\Theta_J^{~b}-\theta_a^{~j}\wedge\theta_j^{~b}
= \eta_{a}^{~j}{}_{\b}^{~\g}   \phi_\g^{~\b}\wedge \theta_{j}^{~b}
-\theta_{a}^{~j} \wedge \eta_{j}^{~b}{}_{\g}^{~\b}   \phi_\b^{~\g}
\mod \phi\wedge\phi
\end{equation}
and \eqref{struc-phi} becomes
\begin{equation}\Label{phi-dif}
(\Psi_a^{~\g} - \psi_{a}^{~\g})\wedge\phi_\g^{~b}
+\phi_{a}^{~\g}\wedge(\3\Psi_{\g}^{~b} - \3\psi_{\g}^{~b})
+\eta_{a}^{~j}{}_{\b}^{~\g}   \phi_\g^{~\b}\wedge \theta_{j}^{~b}
-\theta_{a}^{~j} \wedge \eta_{j}^{~b}{}_{\g}^{~\b}   \phi_\b^{~\g}
=0,
\mod \phi\wedge\phi
\end{equation}
where $\phi\wedge\phi$ stands for the space generated by all possible wedge products
$\phi_{\b}^{~\g}\wedge\phi_{\delta}^{~\tau}$.
Together with \eqref{phi-dif}
we shall consider the structure equations
obtained by differentiating \eqref{eq1} and using \eqref{seq},
\eqref{dtheta0-1} and \eqref{dtheta0}:
\begin{align}
\eta_{a}^{~J}{}_{\b}^{~\g}
\theta_\gamma^{~k}\wedge\theta_k^{~\beta}  =
(\Psi_a^{~\b} - \psi_{a}^{~\b}) \wedge \theta_\b^{~J}
+\theta_a^{~k}\wedge (\Omega_k^{~J} -\omega_{k}^{~J})
 \mod \phi. \Label{dtheta}
\end{align}

\subsection{Determination of $\Theta_{a}^{~J}$ for $a>q$, $J>n$.}
In case $a>q$ and $J>n$ the right-hand side of \eqref{dtheta} is zero.
Since the forms $\theta_{\g}^{~k}$ and $\theta_{k}^{~\b}$
are $(1,0)$ and $(0,1)$ respectively and are linearly independent,
 we conclude
\begin{equation*}
\eta_{a}^{~J}{}_{\b}^{~\g}=0, \quad a>q, \ J>n,
\end{equation*}
and hence \eqref{eq1} yields
\begin{equation}\Label{theta4}
\Theta_a^{~J}=0,  \quad a>q, \ J>n.
\end{equation}

\subsection{Determination of $\Theta_{a}^{~j}$ for $a>q$.}
For $a>q$, $b=\b\le q$, \eqref{phi-dif} takes the form
\begin{equation}\Label{phi-psi}
0=\Psi_a^{~\a}\wedge\phi_\a^{~\b}
+\eta_{a}^{~j}{}_{\delta}^{~\a}
\phi_\a^{~\delta}
\wedge \theta_j^{~\b}
~\text{   mod   }~\phi\wedge\phi, \quad a>q.
\end{equation}
Since the forms $\Re(\phi_\beta^{~\alpha})$, $\a<\b$, and $\Im (\phi_\beta^{~\alpha}),~\a\le\b$, are linearly independent over $\R$, it follows that $\phi_{\a}^{~\b}$ are linearly independent over $\C$.
Then we can collect the coefficients in front of these forms and  apply complex Cartan's Lemma for  a fixed $\b$ to obtain
$$\Psi_a^{~\a}\in {\rm span}\{ \theta_j^{~\b},\,\phi\}, \quad a>q.$$
 But $\Psi_a^{~\a}$ is independent of the choice of $\b$.
Therefore, since $\theta_{j}^{~\b}$ are linearly independent
and we have assumed $q\ge 2$, we conclude that
\begin{equation}\Label{psi-sigma}
\Psi_a^{~\a}=0~\text{   mod   }~\phi, \quad a>q,
\end{equation}
and hence \eqref{phi-psi} implies
$$\eta_{a}^{~k}{}_{\a}^{~\g}\phi_\gamma^{~\a}=0, \quad a>q,$$
which in view of \eqref{eq1} yields
\begin{equation}\Label{theta3}
\Theta_a^{~j}=0, \quad a>q.
\end{equation}

\subsection{Reducing the freedom for $\Theta_{\a}^{~J}$ for $J>n$}
Next we use \eqref{dtheta} in case $a=\a\le q$ and $J>n$  that becomes
\begin{equation}\Label{nonum}
\theta_\alpha^{~k}\wedge\Omega_k^{~J}=\eta_{\a}^{~J}{}_{\b}^{~\g}
\theta_\gamma^{~k}\wedge\theta_k^{~\beta} \mod \phi, \quad J>n.
\end{equation}
Since $\theta_{\a}^{~k}$ are linearly independent and $(1,0)$,
whereas $\theta_k^{~\beta}$ are $(0,1)$,
the term with $\theta_\gamma^{~k}\wedge\theta_k^{~\beta}$, $\g\ne\a$, in the right-hand side cannot occur in the left-hand side. Therefore
$$
\eta_{\alpha~\beta}^{~J~\gamma}=0~\text{   if   }~\g\ne\a, \quad J>n,$$
and hence \eqref{eq1} becomes
\begin{equation}\Label{eq5}
\Theta_\alpha^{~J}=\eta_{\alpha~\beta}^{~J}\phi_\alpha^{~\beta}, \quad J>n,
\end{equation}
where
$$\eta_{\alpha~\beta}^{~J}:=\eta_{\alpha~\beta}^{~J~\a},$$
and  \eqref{nonum} becomes
\begin{equation}\Label{nonum1}
\theta_\alpha^{~k}\wedge\Omega_k^{~J}=\eta_{\a~\b}^{~J}
\theta_\a^{~k}\wedge\theta_k^{~\beta} \mod \phi, \quad J>n,
\end{equation}
i.e.\
\begin{equation}\Label{nonum2}
\theta_\alpha^{~k}\wedge (
\Omega_k^{~J}-\eta_{\a~\b}^{~J}\theta_k^{~\beta}
) =0\mod \phi, \quad J>n.
\end{equation}
Then  using linear independence of $\theta_{\a}^{~k}$ and applying Cartan's Lemma, we obtain
\begin{equation}\Label{omegak}
\Omega_k^{~J}=\eta_{\alpha~\beta}^{~J}\theta_k^{~\beta}\mod
\{\phi, ~\theta_\alpha\}, \quad J>n.
\end{equation}
Since $\Omega_k^{~J}$ is independent of $\alpha$ and $q\ge2$ by our assumption, taking $(0,1)$ parts we obtain
\begin{equation}\Label{eta-J}
\eta_{\alpha~\beta}^{~J}=\eta_{~\b}^{J},
\quad J>n,
\end{equation}
for some $\eta_{~\b}^{J}$,
hence \eqref{eq5} implies
\begin{equation}
\Theta_\alpha^{~J}=\eta_{~\beta}^{J}\phi_\alpha^{~\beta}, \quad J>n,
\end{equation}
and \eqref{omegak} yields
\begin{equation}\Label{omega-free}
\Omega_k^{~J}=\eta_{~\beta}^{J}\theta_k^{~\beta}~\text{   mod   }~\phi, \quad J>n,
\end{equation}
where we dropped $\theta_{\a}$ on the right-hand side, 
since now both sides are independent of $\a$ and since $q\ge 2$ by our assumption.

\subsection{Reducing the freedom for $\Theta_{\a}^{~j}$}
Here we use the structure equation \eqref{phi-dif} in case $a=\a\le q$ and $b=\b\le q$.
Then the forms $\phi_\alpha^{~\gamma}$, $\g\ne\b$,
and $\phi_\delta^{~\beta}$, $\delta\ne\a$, together with $\phi_{\a}^{~\b}$  are linearly independent. Therefore identifying the coefficients in front of these forms
in \eqref{phi-dif}
and using Cartan's Lemma yields
\begin{align}
\Psi_\alpha^{~\gamma}-\psi_\alpha^{~\gamma}&=0
\mod \{\theta_\a,\1{\theta_\b},\,\phi\}~\text{   if   }~ ~\gamma\neq\alpha,\\
\left(\Psi_\alpha^{~\alpha}-\psi_\alpha^{~\alpha}\right)
-(\3\Psi_{\b}^{~\b}-\3\psi_{\b}^{~\b})
&=0\mod \{\theta_\a,\1{\theta_\b},\,\phi\}.\Label{phi-ab}
\end{align}
Since $\left(\Psi_\alpha^{~\gamma}-\psi_\alpha^{~\gamma}\right),~\gamma\neq\alpha,$ is independent of the choice of $\beta$ and $q\ge2$, we conclude
\begin{equation}\Label{psi-psi}
\Psi_\alpha^{~\gamma}-\psi_\alpha^{~\gamma} =0
\mod \{\theta_\alpha,\phi\} ~\text{   if   }~\gamma\neq\alpha.
\end{equation}

Substituting now \eqref{psi-psi} into \eqref{dtheta} for $a=\a\le q$, 
$J=j\le n$, 
 and identifying coefficients in front of $\theta_{\g}^{~k}\wedge\theta_{k}^{~\b}$
 for $\g\ne\a$, since $\theta_{\g}^{~j}$ are $(1,0)$ and linearly independent, whereas $\theta_{j}^{~\a}= -\1{\theta_{\a}^{~j}}$ are $(0,1)$,
we obtain
$$\eta_{\alpha~\beta}^{~j~\gamma}=0~\text{   if   }\gamma\neq\alpha.$$
In view of
\eqref{psi-psi} we can write
\begin{equation}\Label{psi-alpha}
\Psi_\alpha^{~\g}-\psi_\alpha^{~\g}=h_{\alpha~k}^{~\g}\theta_\alpha^{~k}~\text{   mod   }~\phi, \quad \gamma\neq\alpha,
\end{equation}
for suitable $h_{\alpha~k}^{~\gamma}$ and put
\begin{equation}\Label{h-aa}
h_{\alpha~k}^{~\a}:=0.
\end{equation}
Then \eqref{dtheta} for $a=\a\le q$ and $J=j\le n$ becomes
$$\left\{(\Psi_\alpha^{~\alpha}-\psi_\alpha^{~\alpha})\delta_k^{~j}-
(\Omega_k^{~j}-\omega_k^{~j})\right\}\wedge\theta_\alpha^{~k}
+\sum_{\g\neq\alpha}(\Psi_\alpha^{~\g}-\psi_\alpha^{~\g})\wedge\theta_\g^{~j}
=\eta_{\alpha~\g}^{~j}\theta_\alpha^{~k}\wedge\theta_k^{~\g}
~\text{   mod   }~\phi,$$
where
$$\eta_{\alpha~\beta}^{~j}:=\eta_{\alpha~\beta}^{~j~\alpha}.$$
Substituting \eqref{psi-alpha}  yields
$$\big\{(\Psi_\alpha^{~\alpha}-\psi_\alpha^{~\alpha})\delta_k^{~j}-
(\Omega_k^{~j}-\omega_k^{~j} 
+ h_{\alpha~k}^{~\gamma}\theta_\gamma^{~j}
-\eta_{\alpha~\g}^{~j}\theta_k^{~\g}
)\big\}\wedge\theta_\alpha^{~k}
=0
~\text{   mod   }~\phi.$$
Now using Cartan's Lemma we obtain
\begin{equation}\Label{psi-omega}
(\Psi_\alpha^{~\alpha}-\psi_\alpha^{~\alpha})\delta_k^{~j}-(\Omega_k^{~j}-\omega_k^{~j})
=-h_{\alpha~k}^{~\gamma}\theta_\gamma^{~j}
+\eta_{\alpha~\gamma}^{~j}\theta_k^{~\gamma}
 \mod  \{\phi,\theta_\alpha\}.
\end{equation}
As consequence of \eqref{phi-ab} and \eqref{symmetries} we also have
\begin{equation}\Label{psi-ab}
\left(\Psi_\alpha^{~\alpha}-\psi_\alpha^{~\alpha}\right)
+\left(\Psi_{\bar\b}^{~\bar\b}-\psi_{\bar\b}^{~\bar\b}\right)
=0\mod \{\theta_{\a},\1{\theta_{\b}},\phi\},
\end{equation}
and, in particular,
\begin{equation}\Label{psi-aa}
\Re\left(\Psi_\alpha^{~\alpha}-\psi_\alpha^{~\alpha}\right)
=0\mod \{\theta_{\a},\1{\theta_{\a}},\phi\}.
\end{equation}
Furthermore, since $\Omega_j^{~k}$ and $\omega_j^{~k}$ are antihermitian
in view of \eqref{symmetries},
taking hermitian part with respect to $(j,k)$ of \eqref{psi-omega},
using \eqref{psi-ab} 
and identifying coefficients in front of $\theta_{k}^{~\g}=-\1{\theta_{\g}^{~k}}$, we obtain
\begin{equation}\Label{h-eta}
\overline{ h_{\alpha~j}^{~\gamma}}+\eta_{\alpha~\gamma}^{~j}=0,~\text{   if   }~\gamma\neq\alpha.
\end{equation}
Next, since $\Omega_j^{~k}$ and $\omega_j^{~k}$ are antihermitian,
 \eqref{psi-ab} implies
$$
(\Psi_\alpha^{~\alpha}-\psi_\alpha^{~\alpha})-(\Omega_j^{~j}-\omega_j^{~j})=
-(\Psi_{\bar\beta}^{~\bar\beta}-\psi_{\bar\beta}^{~\bar\beta})+(\Omega_{\bar j}^{~\bar j}-\omega_{\bar j}^{~\bar j})
~\text{   mod   }~ \{\phi,\theta_\alpha, \1{\theta_{\beta}}\}.
$$
Hence using \eqref{psi-omega} for $k=j$ and adding its conjugate with $\a$ replaced by $\b$, using \eqref{h-eta} and identifying
the coefficients in front of $\theta_{j}^{~\g}$
we obtain
$$\eta_{\alpha~\gamma}^{~j}=\eta_{\beta~\gamma}^{~j}, \quad \g\ne\b,$$
and hence
$$\eta_{\a}^{~j}{}_{\g}=\eta^{j}_{~\g}$$
for suitable $\eta^{j}_{~\g}$.
Hence \eqref{eq1} implies
\begin{equation}\Label{theta-eta}
\Theta_\alpha^{~j}-\theta_\alpha^{~j}=\eta_{~\beta}^{j}\phi_\alpha^{~\beta}.
\end{equation}

\subsection{Determination of $\Theta_{a}^{~J}$ by a change of frame}
Let $(B_a^{~J})$ be  a matrix defined by
$$B_\alpha^{~J}:=\eta_{~\alpha}^{J}, \quad B_{a}^{~J}:=0, \quad a>q,
$$
where $\eta_{~\alpha}^{J}$ is defined by \eqref{eta-J} for $J>n$.
Consider the change of frame of $S_{p',q'}$
discussed after Definition~\ref{changes},
given by
\begin{equation*}
\widetilde Z_a=Z_a,\quad
\widetilde X_J=X_{J}  + C_J^{~b}Z_b,\quad
\widetilde Y_a=Y_a+A_a^{~b}Z_b+B_a^{~J}X_J
\end{equation*}
such that
$$C_J^{~a}:=-B_J^{~a}$$
and $A_{a}^{~b}$ satisfies
$$\left(A_a^{~b}+\overline{A_b^{~a}}\right)+\sum_{J}B_a^{~J} \1{B_{b}^{~J}}=0.$$
Since the sum here is hermitian, one can always choose $A_{a}^{~b}$ with this property.
Then $\Phi_{a}^{~b}=\phi_{a}^{b}$ remain the same while $\Theta_{a}^{J}$ change to
$$\Theta_a^{~J}-\phi_a^{~b}B_b^{~J}.$$
Therefore \eqref{theta-eta} becomes
\begin{equation}\Label{theaj}
\Theta_a^{~J}=\theta_a^{~J},
\end{equation}
which is equivalent to $\eta^{J}_{~\b}=0$ and hence
$\eta_{\a}^{~j}{}_{\g}=0$ and $h_{\alpha~j}^{~\gamma}=0$ in view of
\eqref{h-aa} and
 \eqref{h-eta}.
Therefore \eqref{psi-omega} implies
\begin{equation}\Label{unname}
(\Psi_\alpha^{~\alpha}-\psi_\alpha^{~\alpha})\delta_{k}^{~j}-(\Omega_k^{~j}-\omega_k^{~j})=0 \mod   \{\phi,\theta_\alpha\},
\end{equation}
and \eqref{psi-alpha} together with \eqref{psi-sigma} implies
\begin{equation}\Label{psi-ag}
\Psi_a^{~\gamma}-\psi_a^{~\gamma}=0~\text{   mod   }~\phi,~\text{   if   }~\gamma\neq a,
\end{equation}
and, since the left-hand side of \eqref{unname} is independent of $\a$ for $j\ne k$,
together with \eqref{omega-free} we obtain
\begin{equation}\Label{omega-modphi}
\Omega_k^{~J}-\omega_k^{~J}=0~\text{   mod   }~\phi,~\text{   if   }~J\neq k.
\end{equation}

\section{Determination of $\Psi_{a}^{~\b}$ and $\Omega_{k}^{~J}$}
Next, we use \eqref{theaj} in the structure equations   for  $d\theta_{a}^{~J}=d\Theta_{a}^{~J}$, which yield
\begin{equation}\Label{struct-cor}
(\Psi_{a}^{~\b} - \psi_{a}^{~\b})\wedge \theta_{\b}^{~J}
+\theta_{a}^{~k}\wedge (\Omega_{k}^{~J} - \omega_{k}^{~J})
+ \phi_{a}^{~\b}\wedge (\Sigma_{\b}^{~J} - \sigma_{\b}^{~J})
= 0.
\end{equation}

\subsection{Determination of $\Psi_{a}^{~\b}$ for $a>q$}
In case $a>q$, $J=j\le n$, \eqref{struct-cor} takes the form
\begin{equation}\Label{psi-a}
\Psi_a^{~\beta}\wedge\theta_\beta^{~j}=0, \quad a>q.
\end{equation}
Together with \eqref{psi-ag}, this yields
\begin{equation}\Label{aq}
\Psi_a^{~\beta}=0, \quad a>q.
\end{equation}

\subsection{Determination of $\Omega_{k}^{~J}$ and $\Sigma_{\b}^{~J}$ for $J>n$}
Next use \eqref{struct-cor} for $a=\a\le q$ and $J>n$ to obtain
$$\theta_\alpha^{~k}\wedge\Omega_k^{~J}
+ \phi_{\a}^{~\b} \wedge \Sigma_{\b}^{~J}=0, \quad J>n.
$$
By Cartan's Lemma,
\begin{equation}
\Omega_{k}^{~J} = 0 \mod \{\theta_{\a}, \phi_{\a}\}, \quad
\Sigma_{\b}^{~J} = 0 \mod \{\theta_{\a}, \phi_{\a}\}, \quad
J>n,
\end{equation}
where $\phi_{\a}$ is the ideal generated by $\phi_{\a}^{~\b}$ for $\a$ fixed.
Since $\Omega_k^{~J}$ and $\Sigma_{\b}^{~J}$ are independent of $\alpha$ and $q\ge2$, we conclude
\begin{equation}
\Omega_k^{~J}=\Sigma_{\b}^{~J}=0, \quad J>n.
\end{equation}

\medskip
We summarize the obtained alignment of the connection forms:

\begin{proposition}\Label{eds-2}
For any local CR embedding $f$ from $S_{p,q}$ into $S_{p',q'}$,
there is a choice of sections of the frame bundles
$\6B_{p,q}\to S_{p,q}$ and $\6B_{p',q'}\to S_{p',q'}$
such that
\begin{align}
\Phi_a^{~b}&=\phi_a^{~b}, \quad
\Theta_a^{~J}=\theta_a^{~J}, \\
\Psi_a^{~\beta}&=0,\quad
\Omega_j^{~K}=0,\quad
\Sigma_{\a}^{~K}=
0, \quad a>q, \quad K>n.
\end{align}
\end{proposition}
\medskip

\begin{remark}
Similarly to Remark~\ref{back1},
we can restrict to changing only the section of the second bundle
$\6B_{p',q'}\to S_{p',q'}$.
\end{remark}

\section{Embeddability in a plane of suitable dimension}

\begin{proposition}\Label{flatness}
Under the assumptions of Theorem~\ref{main}, 
there exist a $(p+q)$-dimensional subspace $V_1$
 and a $(q'-q)$-dimensional subspace $V_2$ in $\mathbb{C}^{p'+q'}$
 with $V_{1}\cap V_{2}=0$
 and such that $\langle\cdot,\cdot\rangle$ nondegenerate of signature $(p,q)$
 when restricted to $V_{1}$ and null when restricted to $V_{2}$
 such that 
$f(S_{p,q})\subset Gr(V_1, q)\oplus V_2$.
\end{proposition}

\begin{proof}
Denote by $M\subset S_{p,q}$ the open subset where $f$ is defined.
Let $Z, X, Y$ be constant vector fields of $\mathbb{C}^{p'+q'}$ forming a $S_{p',q'}$-frame adapted to $M$ at a fixed reference point in $M$. Let
\begin{align}
\widetilde Z_a&=\lambda_a^{~b}Z_b+\eta_a^{~K}X_K+\zeta_a^{~b}Y_b, \Label{adapted1}\\
\widetilde X_J&=\lambda_J^{~b}Z_b+\eta_J^{~K}X_K+\zeta_J^{~b}Y_b,\Label{adapted2}\\
\widetilde Y_a&=\tilde\lambda_a^{~b}Z_b+\tilde\eta_a^{~K}X_K+\tilde\zeta_a^{~b}Y_b\Label{adapted3}
\end{align}
be an adapted $S_{p',q'}$-frame along $M$.
Write
\begin{equation}\Label{A}
A
=
\begin{pmatrix}
\lambda_a^{~b} & \eta_a^{~K} & \zeta_a^{~b}\\
\lambda_J^{~b} & \eta_J^{~K} & \zeta_J^{~b}\\
\tilde\lambda_a^{~b} & \tilde\eta_a^{~K} & \tilde\zeta_a^{~b}\\
\end{pmatrix},
\end{equation}
so that \eqref{adapted1} - \eqref{adapted3} take the form
\begin{equation}\Label{A-eq}
\begin{pmatrix}
\2Z\\ \2X\\ \2Y
\end{pmatrix}
=
A
\begin{pmatrix}
Z\\ X\\ Y
\end{pmatrix}
.
\end{equation}
Since $Z, X, Y$ form an adapted frame at a reference point of $M$, we may assume that
\begin{equation}\Label{initial}
A=I_{p'+q'}
\end{equation}
at the reference point.
Since $Z, X, Y$ are constant vector fields, i.e., $dZ=dX=dY=0$, 
differentiating \eqref{A-eq} and using \eqref{differential}
we obtain
\begin{equation}\Label{dA}
dA=\Pi A,
\end{equation}
where $\Pi$ is the connection matrix of $S_{p',q'}$, i.e.\
\begin{equation}\Label{dA-expand}
dA=
\begin{pmatrix}
\Psi_{a}^{~b} & \Theta_{a}^{~J} & \Phi_{a}^{~b}\\
\Sigma_{K}^{~b} & \Omega_{K}^{~J} & \Theta_{K}^{~b}\\
\Xi_{a}^{~b} & \Sigma_{a}^{~J} & \3\Psi_{a}^{~b}\\
\end{pmatrix}
A.
\end{equation}

Next, it follows from Proposition~\ref{eds-2} that
\begin{equation}\Label{dZa}
d\2Z_a= \sum_{b>q}\Psi_{a}^{~b} \2Z_{b}, \quad a>q,
\end{equation}
in particular, 
the span of $\2Z_{a}$, $a>q$, 
is independent of the point in $M$.
Hence together with \eqref{adapted1} and \eqref{initial}, we conclude
\begin{equation}\Label{Z_a-determined}
\eta_a^{~K}=\zeta_a^{~b}=0, \quad a>q.
\end{equation}

Furthermore, \eqref{dA-expand} implies
\begin{equation}
\begin{pmatrix}
d\eta_{a}^{~K} \\ d\eta_{J}^{~K} \\ d\tilde\eta_{a}^{~K}
\end{pmatrix}
=
\begin{pmatrix}
\Psi_{a}^{~b} & \Theta_{a}^{~L} & \Phi_{a}^{~b}\\
\Sigma_{J}^{~b} & \Omega_{J}^{~L} & \Theta_{J}^{~b}\\
\Xi_{a}^{~b} & \Sigma_{a}^{~L} & \3\Psi_{a}^{~b}\\
\end{pmatrix}
\begin{pmatrix}
\eta_{b}^{~K} \\ \eta_{L}^{~K} \\ \tilde\eta_{b}^{~K}
\end{pmatrix}.
\end{equation}
In particular, restricting to $a=\a\le q$ and $J=j\le n$ and using 
Proposition~\ref{eds-2} (together with the symmetry relations analogous to \eqref{symmetries}) we obtain
\begin{equation}\Label{d-eta}
\begin{pmatrix}
d\eta_{\a}^{~K} \\ d\eta_{j}^{~K} \\ d\tilde\eta_{\a}^{~K}
\end{pmatrix}
=
\begin{pmatrix}
\Psi_{\a}^{~b} & \theta_{\a}^{~L} & \phi_{\a}^{~b}\\
\Sigma_{j}^{~b} & \Omega_{j}^{~L} & \theta_{j}^{~b}\\
\Xi_{\a}^{~b} & \Sigma_{\a}^{~L} & \3\Psi_{\a}^{~b}\\
\end{pmatrix}
\begin{pmatrix}
\eta_{b}^{~K} \\ \eta_{L}^{~K} \\ \tilde\eta_{b}^{~K}
\end{pmatrix}.
\end{equation}
Now with \eqref{Z_a-determined} and Proposition~\ref{eds-2} taken into account,  \eqref{d-eta} becomes
\begin{equation}
\begin{pmatrix}
d\eta_{\a}^{~K} \\ d\eta_{j}^{~K} \\ d\tilde\eta_{\a}^{~K}
\end{pmatrix}
=
\begin{pmatrix}
\Psi_{\a}^{~\b} & \theta_{\a}^{~l} & \phi_{\a}^{~\b}\\
\Sigma_{j}^{~\b} & \Omega_{j}^{~l} & \theta_{j}^{~\b}\\
\Xi_{\a}^{~\b} & \Sigma_{\a}^{~l} & \3\Psi_{\a}^{~\b}\\
\end{pmatrix}
\begin{pmatrix}
\eta_{\b}^{~K} \\ \eta_{l}^{~K} \\ \tilde\eta_{\b}^{~K}
\end{pmatrix}.
\end{equation}
Repeating the above argument for $\zeta$ instead of $\eta$,
we obtain
\begin{equation}
\begin{pmatrix}
d\zeta_{\a}^{~b} \\ d\zeta_{j}^{~b} \\ d\tilde\zeta_{\a}^{~b}
\end{pmatrix}
=
\begin{pmatrix}
\Psi_{\a}^{~\b} & \theta_{\a}^{~l} & \phi_{\a}^{~\b}\\
\Sigma_{j}^{~\b} & \Omega_{j}^{~l} & \theta_{j}^{~\b}\\
\Xi_{\a}^{~\b} & \Sigma_{\a}^{~l} & \3\Psi_{\a}^{~\b}\\
\end{pmatrix}
\begin{pmatrix}
\zeta_{\b}^{~b} \\ \zeta_{l}^{~b} \\ \tilde\zeta_{\b}^{~b}
\end{pmatrix}.
\end{equation}
Thus each of the vector valued functions 
$\eta^K:=(\eta_\alpha^{~K}, \eta_j^{~K}, \tilde\eta_\alpha^{~K})$ for
a fixed $K$  and 
$\zeta^b:=(\zeta_\alpha^{~b}, \zeta_j^{~b}, \tilde\zeta_\alpha^{~b})$ for  a fixed 
$b$
 satisfies a complete system of linear first order differential equations.
Then by the initial condition \eqref{initial} and the uniqueness of solutions, we conclude, in particular, that
\begin{equation}
\eta^{~K}=\zeta^{~b}=0,
\quad K>n, \, b>q.
\end{equation}
Hence \eqref{A-eq} implies
\begin{align}
\widetilde Z_\alpha&=\lambda_\alpha^{~b}Z_b+\eta_\alpha^{~k}X_k+\zeta_\alpha^{~\beta}Y_\beta.
\Label{Z-al} 
\end{align}

Now setting
\begin{equation}
\3Z_{\a}:= \2Z_{\a} - \sum_{b>q} \l_{\a}^{b} Z_{b},
\end{equation}
we still have
\begin{equation}
\sp \{ \3Z_{\a}, \2Z_{q+1},\ldots, \2Z_{q'}\} =  \sp \{ \2Z_{a} \},
\end{equation}
whereas \eqref{Z-al} becomes
\begin{equation}
\3Z_\alpha=\lambda_\alpha^{~\beta}Z_\beta+\eta_\alpha^{~k}X_k+\zeta_\alpha^{~\beta}Y_\beta,
\end{equation}
implying
$$\sp \{\3Z_{\a}\} \subset \sp  \{ Z_1,\ldots,Z_q,X_1,\ldots,X_n, Y_1,\ldots,Y_q\}.
$$
Then together with 
\eqref{dZa} we conclude that
$$f(M)=\sp\{ \2Z_{a}\}  
= \sp\{\3Z_{\a}\} \oplus \sp\{\2Z_{q+1},\ldots, \2Z_{q'} \}  $$
$$ = \sp\{\3Z_{\a}\} \oplus \sp\{Z_{q+1},\ldots, Z_{q'} \}  
 \subset Gr(V_1, q)\oplus V_2,$$
 where 
\begin{equation}
V_1=\sp\{ Z_1,\ldots,Z_q,X_1,\ldots,X_n, Y_1,\ldots,Y_q\},\quad
V_2=\sp\{Z_{q+1},\ldots,Z_{q'}\}.
\end{equation}
\end{proof}

\section{Rigidity of CR embeddings from $S_{p,q}$ to $S_{p',q'}$}

As consequence of Proposition~\ref{flatness}, we conclude that,
after a linear change of coordinates, 
$f(M)$ locally coincides with $S_{p,q}$ linearly embedded into $S_{p',q'}$.
Identifying $M=S_{p,q}$ with its image, $f$
becomes a local CR-automorphism of $M$.
Then by a theorem of Kaup-Zaitsev \cite[Theorem~4.5]{KZ06},
$f$ is a restriction of a global CR-automorphism of $S_{p,q}$.
Furthermore, by \cite[Theorem~8.5]{KZ00},
$f$ extends to a biholomorphic automorphism
of the bounded symmetric domain
and the rigidity follows.

\end{document}